\documentclass[12pt]{article}
\usepackage{amsthm}
\usepackage{amsfonts}
\usepackage{latexsym}
\usepackage{amsmath}
\usepackage{amssymb}
\usepackage{amscd}
\usepackage{epsfig}
\usepackage{graphicx, subfigure}
\usepackage{color}
\addtocounter{MaxMatrixCols}{4}
\theoremstyle{plain}
\newtheorem{theorem}{Theorem}[section]

\newtheorem{lemma}[theorem]{Lemma}

\theoremstyle{definition}

\newtheorem{conjecture}[theorem]{Conjecture}
\begin{document}

\title{Non-Hamiltonian 3--Regular Graphs with Arbitrary Girth}
\author{M. Haythorpe\footnote{Flinders University. Email: michael.haythorpe@flinders.edu.au}}
%\institute{M. Haythorpe
%\at Flinders University\\
%\email{michael.haythorpe@flinders.edu.au}
%}
\maketitle

\begin{abstract}It is well known that 3--regular graphs with arbitrarily large girth exist. Three constructions are given that use the former to produce non-Hamiltonian 3--regular graphs without reducing the girth, thereby proving that such graphs with arbitrarily large girth also exist. The resulting graphs can be 1--, 2-- or 3--edge-connected depending on the construction chosen. From the constructions arise (naive) upper bounds on the size of the smallest non-Hamiltonian 3--regular graphs with particular girth. Several examples are given of the smallest such graphs for various choices of girth and connectedness.\end{abstract}

\section{Introduction}\label{sec-intro}

Consider a $k$--regular graph $\Gamma$ with girth $g$, containing $N$ vertices. Then $\Gamma$ is said to be a {\em $(k,g)$--cage} if and only if all other $k$--regular graphs with girth $g$ contain $N$ or more vertices. The study of cages, or {\em cage graphs}, goes back to Tutte \cite{tutte}, who later gave lower bounds on $n(k,g)$, the number of vertices in a $(k,g)$--cage \cite{tutte2}. The best known upper bounds on $n(k,g)$ were obtained by Sauer \cite{sauer} around the same time. Since that time, the advent of vast computational power has enabled large-scale searches such as those conducted by McKay et al \cite{mckay}, and Royle \cite{royle} who maintains a webpage with examples of known cages. Erd\H{o}s and Sachs first gave a proof that 3--regular graphs of arbitrary girth exist \cite{erdos} which was converted into an algorithm to construct such graphs by Biggs \cite{biggs} in 1998.

This manuscript focuses solely on 3--regular graphs, and the additional requirement of {\em non-Hamiltonicity} is added. A graph containing $N$ vertices is {\em non-Hamiltonian} if and only if it does not contain any simple cycles of length $N$. Such graphs are thought to be rare in general, and this has been proven to be the case for regular graphs \cite{robinson}. First, three constructions will be given that produce non-Hamiltonian 3--regular graphs with chosen girth $g$, and it will be shown that these constructions produce graphs that are 1--, 2-- or 3--edge-connected respectively. Then the concept of a $(k,g)$--cage will be extended to a {\em $(k,g,e)$--prison}, being an $e$--edge-connected, non-Hamiltonian $k$--regular graph with girth $g$ containing the minimum possible number of vertices. Naive upper bounds arising from the constructions will be given for 3--regular prisons. The manuscript will conclude with some examples of such prisons for small girths.

\section{Bridge construction}\label{sec-bridge}

Recall that, for any desired $g \geq 3$, Hamiltonian 3--regular graphs can be produced with girth $g$. Call such a graph $\Gamma_g$. Then, select an edge $e$ from $\Gamma_g$ such that there is at least one cycle of length $g$ in $\Gamma_g$ that does not traverse edge $e$. It is clear that such an edge must exist. Then, by taking two copies of $\Gamma_g$, it is possible to break edge $e$ in both graphs and insert a bridge to produce a 1--edge-connected, 3--regular graph, say $\Gamma_{g,1}^*$. In the language used in Baniasadi et al \cite{genetictheory}, this construction is equivalent to performing a type 1 breeding operation. An example, for $g = 3$, is given in Figure \ref{fig-con1}.

\begin{figure}[h!]
\centering\includegraphics[scale=0.45]{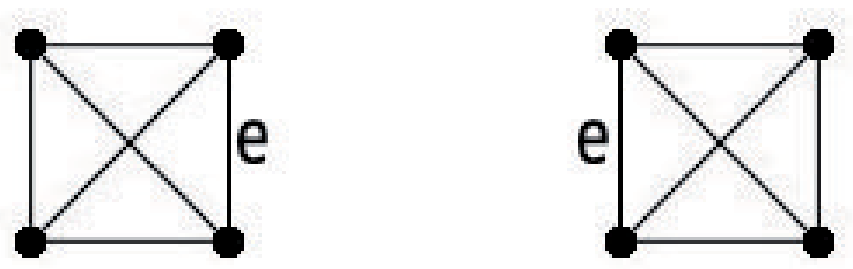} \includegraphics[scale=0.45]{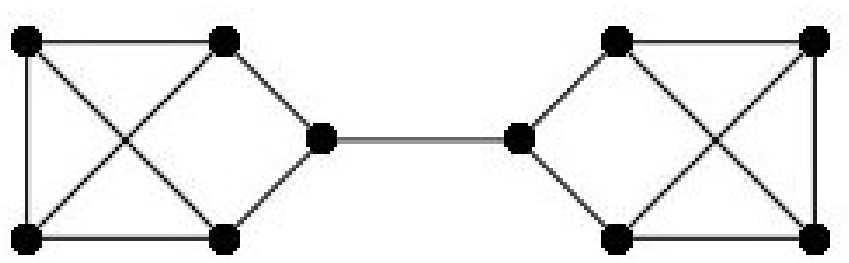}
\caption{Two copies of $\Gamma_3$, and the resulting $\Gamma_{3,1}^*$ produced by the bridge construction.}
\label{fig-con1}
\end{figure}

It is well-known that graphs containing bridges (i.e. 1--edge-connected graphs) are always non-Hamiltonian; indeed, for 3--regular graphs it is conjectured that almost all non-Hamiltonian graphs contain bridges \cite{conjecturepaper}.

\begin{lemma}The resultant graph $\Gamma_{g,1}^*$ has girth $g$.\label{lem-1c}\end{lemma}

\begin{proof}It is clear that the girth of $\Gamma_{g,1}^*$ is at most $g$, since there are cycles of length $g$ that remain untouched by the construction. It then suffices to show that the construction does not introduce any smaller cycles. Clearly no cycles can use the bridge, and the only other alteration caused by the construction is the addition of another vertex to edge $e$. Then, any cycles in $\Gamma_g$ that pass through edge $g$ are now longer by 1 in $\Gamma_{g,1}^*$, but certainly no shorter cycles than those in $\Gamma_g$ have been introduced.\end{proof}

A naive upper bound of $2n(3,g) + 2$ can now be given for the size of the smallest non-Hamiltonian 3--regular graph with girth $g$.

Of course, it is worth noting that while a construction process of this type cannot introduce smaller cycles, if an edge $e$ is selected in $\Gamma_g$ that {\em every} cycle of length $g$ passes through, the construction will produce a graph of girth $g+1$. Certainly 3--regular graphs exist that contain an edge that lies on every shortest cycle; for example, any 3--regular graph containing exactly one triangle has three such edges. It is not clear whether any cages exist that contain an edge that lies on every shortest cycle, however. If no such cages exist then the above upper bound is tight for non-Hamiltonian 3--regular, 1--edge-connected graphs with girth $g$.

Although this manuscript is focused on 3--regular graphs, it is worth noting that the same type of construction could be used on $k$--regular graphs for any $k \geq 3$ and it is fairly easy to check that the result generalises. The same cannot be said for the two constructions in the following section.

\section{Non-bridge constructions}\label{sec-non-bridge}

Although the above construction demonstrates that non-Hamiltonian 3--regular graphs exist with arbitrary girth, the construction can only produce bridge graphs. These graphs are, perhaps, uninteresting as they can be detected in polynomial time. However, for 2--edge-connected or 3--edge-connected 3--regular graphs, determining Hamiltonicity is an NP-complete problem \cite{gareyjohnson}. Two constructions now follow that produce non-Hamiltonian 3--regular graphs with chosen girth $g$ that are 2--edge-connected or $3$--edge-connected respectively.

As in the previous section, consider $\Gamma_g$, a Hamiltonian 3--regular graph with girth $g$, and an edge $e$ from $\Gamma_g$ such that there is at least one cycle of length $g$ in $\Gamma_g$ that does not traverse edge $e$. Then, by taking three copies of $\Gamma_g$, we can construct a new 3--regular, 2--edge-connected graph by the following process. Break edge $e$ in two copies of $\Gamma_g$, and join the broken edges together to form a 2-bond. Then, insert a vertex into both edges of the 2-bond. Finally, break edge $e$ in the third copy of $\Gamma_g$ and connect the two broken parts to the two vertices. Call the resulting graph $\Gamma_{g,2}^*$. In the language used in Basiasadi et al \cite{genetictheory} this construction is equivalent to first performing a type 2 breeding operation, then performing a type 2 parthenogenic operation, and finally performing a second type 2 breeding operation on the parthenogenic bridge. An example, for $g = 3$, is given in Figure \ref{fig-con2}.

\begin{figure}[h!]
\centering\includegraphics[scale=0.45]{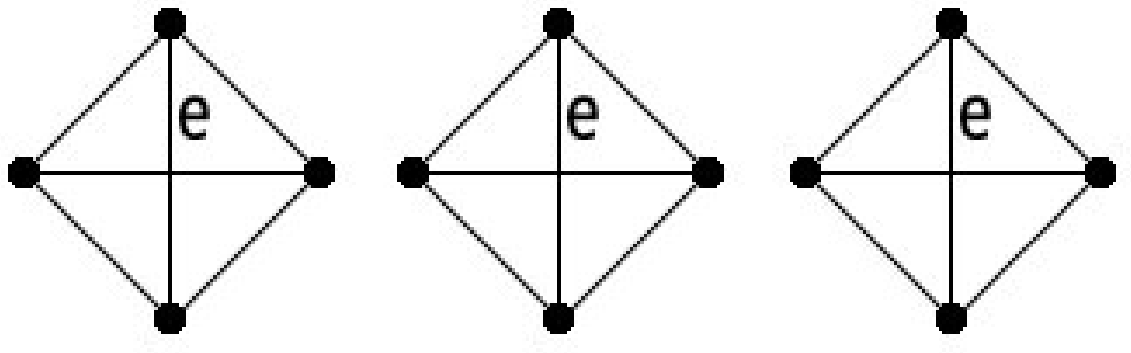} \includegraphics[scale=0.45]{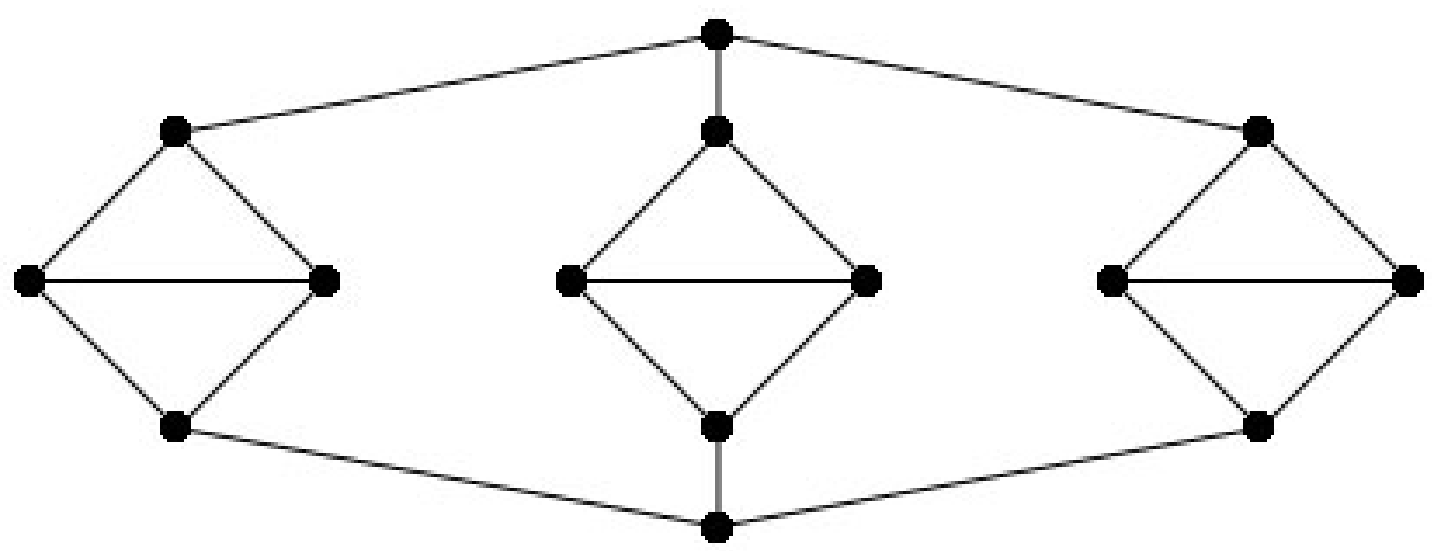}
\caption{Three copies of $\Gamma_3$, and the resulting $\Gamma_{3,2}^*$ produced by the 2--edge-connected construction.}
\label{fig-con2}
\end{figure}

It is easy to see that the resulting graph $\Gamma_{g,2}^*$ must be non-Hamiltonian. Starting from the top introduced vertex, one of the three subgraphs arising from a copy of $\Gamma_g$ must be visited first. Then once it is exited, the bottom introduced vertex is visited. A second subgraph must then be visited, but once it is exited the top introduced vertex must again be visited. Since it is not possible to visit all three subgraphs with a simple cycle, the graph is non-Hamiltonian. An alternative proof of non-Hamiltonicity for graphs of this type is given in Baniasadi and Haythorpe \cite{properties}.

\begin{lemma}The resultant graph $\Gamma_{g,2}^*$ has girth $g$.\label{lem-2c}\end{lemma}

\begin{proof}Similar to the proof of Lemma \ref{lem-1c}, there are cycles of length $g$ that remain untouched by this construction, so it suffices to ensure that the construction does not introduce any smaller cycles. Consider a cycle of length $c$ in $\Gamma_g$ that passes through edge $e$. In $\Gamma_{g,2}^*$ the equivalent cycle must pass out of respective subgraph, through an introduced vertex, through a second subgraph, and back through the other introduced vertex. It is clear that this cycle has length greater than $c$, and so only longer cycles have been.\end{proof}

A naive upper bound of $3n(3,g) + 2$ can now be given for the size of the smallest non-Hamiltonian 3--regular, 2--edge-connected graph with girth $g$. Next, a construction is given that produces non-Hamiltonian 3--regular, 3--edge-connected graphs with chosen girth $g$.

Consider the Petersen graph \cite{petersen}, which is the smallest non-Hamiltonian 3--regular, 3--edge-connected graph. Then, consider again $\Gamma_g$ as above. Select a vertex $v$ in $\Gamma_g$ such that there is at least one cycle of length $g$ in $\Gamma_g$ that does not traverse vertex $v$. Then, by taking ten copies of $\Gamma_g$, we can construct a new 3--regular, 3--edge-connected graph by the following process. For each copy of $\Gamma_g$, remove the vertex $v$, leaving three broken edges. Then, replace a vertex of the Petersen graph with the remaining subgraph of the copy of $\Gamma_g$ such that the three broken edges take the place of the three incident edges of the removed vertex. Call the resulting graph $\Gamma_{g,3}^*$. In the language used in Baniasadi et al \cite{genetictheory} this construction is equivalent to performing ten type 3 breeding operations, starting with the Petersen graph and a copy of $\Gamma_g$, and continuing with the other copies of $\Gamma_g$. An example, for $g = 3$, is given in Figure \ref{fig-con3}.

\begin{figure}[h!]
\centering\includegraphics[scale=0.5]{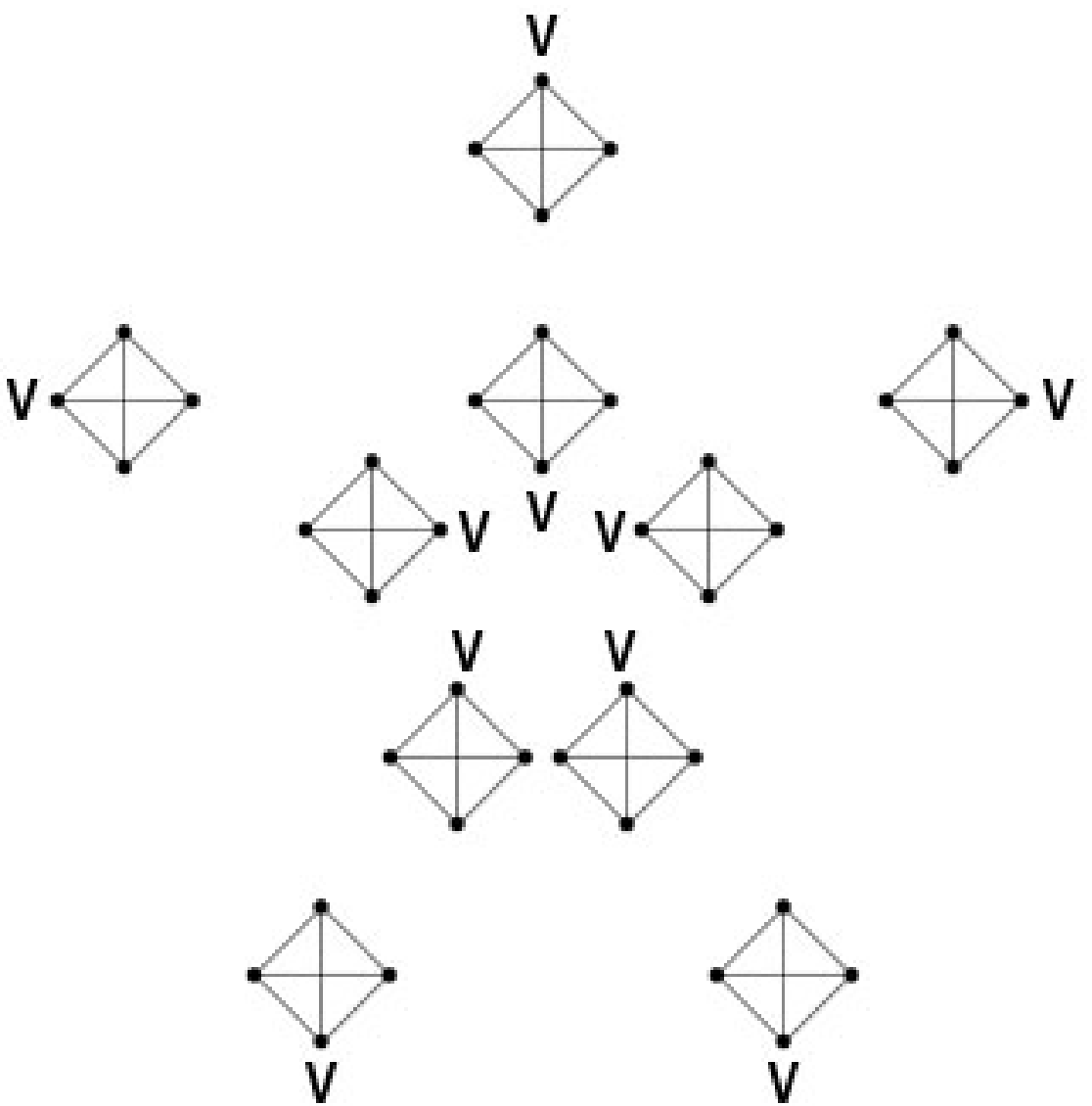} \includegraphics[scale=0.5]{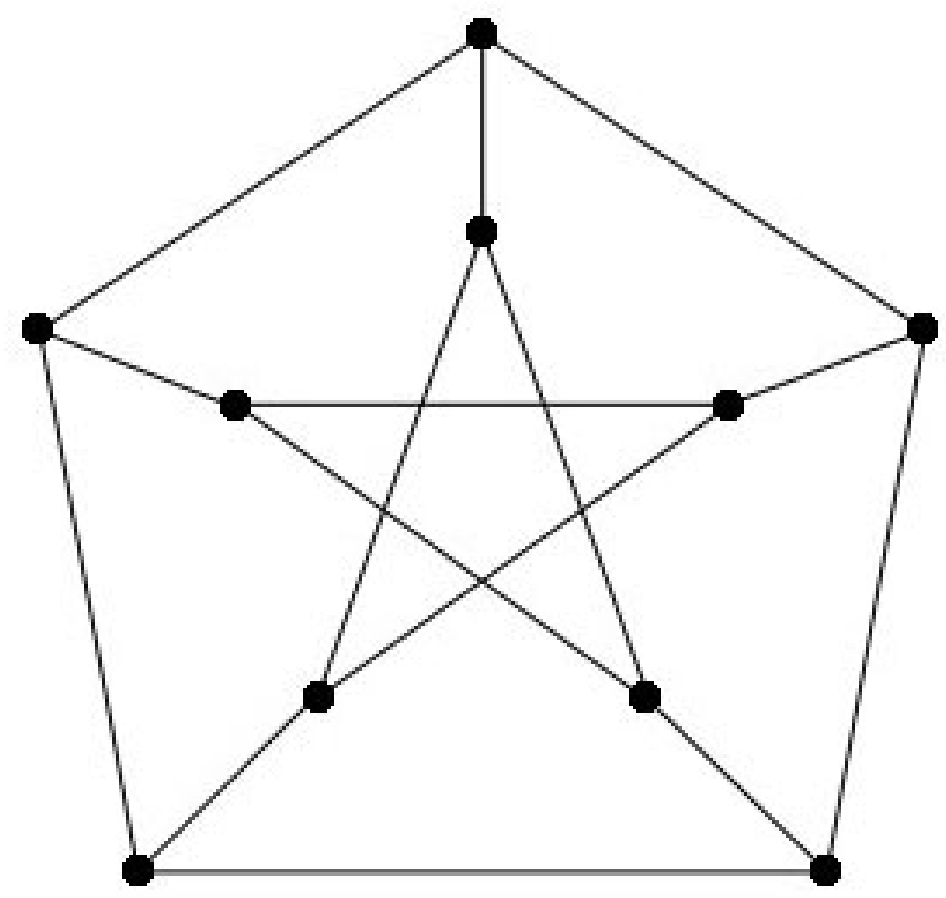}\\
\includegraphics[scale=0.5]{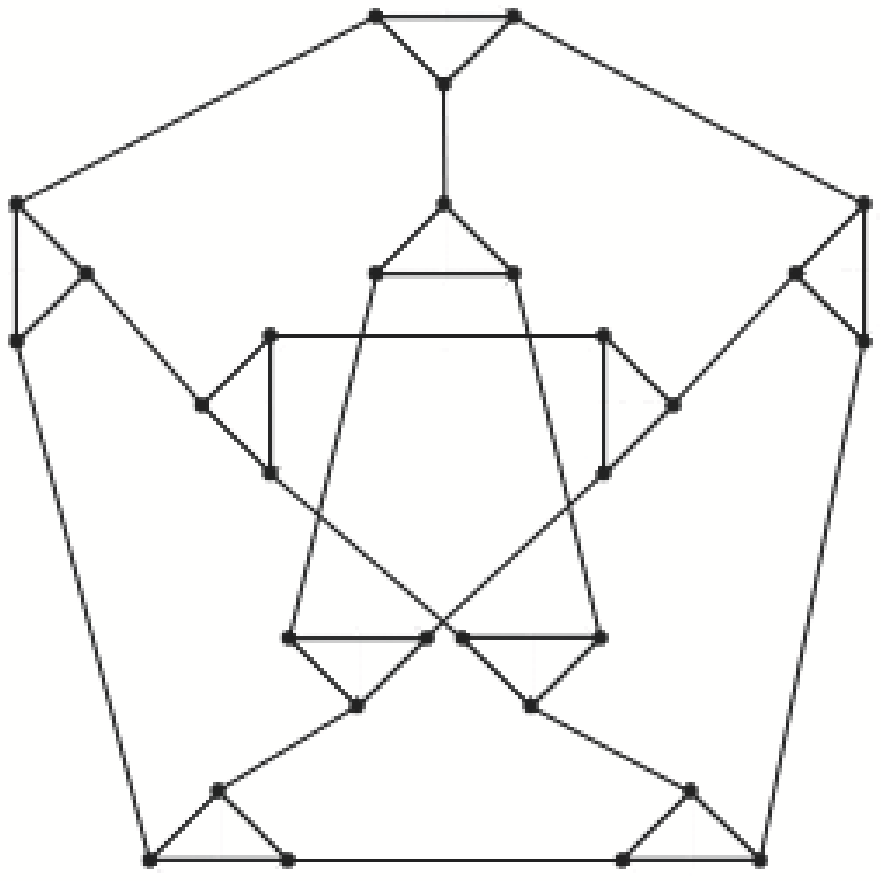}
\caption{Ten copies of $\Gamma_3$, the Petersen graph, and the resulting $\Gamma_{3,3}^*$ produced by the 3--edge-connected construction.}
\label{fig-con3}
\end{figure}

Since each subgraph corresponding to a copy of $\Gamma_g$ is connected to the rest of the graph only by three edges, it is clear that, once entered, the entirety of any given subgraph must be traversed before exiting if a Hamiltonian cycle in $\Gamma_{g,3}^*$ is desired. Therefore, in $\Gamma_{g,3}^*$, each subgraph functions equivalently to a vertex, and therefore the Petersen graph shares its non-Hamiltonicity with $\Gamma_{g,3}^*$. An alternative proof using results Baniasadi and Haythorpe \cite{properties} arises from noticing that $\Gamma_{g,3}^*$ has the non-Hamiltonian Petersen graph as an ancestor gene, violating the necessary condition for $\Gamma_{g,3}^*$ to be Hamiltonian.

\begin{lemma}The resultant graph $\Gamma_{g,3}^*$ has girth $g$.\label{lem-3c}\end{lemma}

\begin{proof}Again, as in the proofs of Lemma \ref{lem-1c} and Lemma \ref{lem-2c}, a cycle of length $g$ remains untouched by this construction, so it suffices to ensure that the construction does not introduce any smaller cycles. Consider a cycle of length $c$ in $\Gamma_g$ that passes through vertex $v$. An equivalent cycle in $\Gamma_g^*$ travels to another subgraph. Even if it were possible to return immediately from the second subgraph back to the first, the resulting cycle would still be of length $c$, and so no shorter cycles have been introduced.\end{proof}

Of course, it is not possible to return immediately from the second subgraph, and in fact would require a path of length at least $c-1$ to reach a vertex that would allow a return to the original subgraph.

A naive upper bound of $10n(3,g) - 10$ can now be given for the size of the smallest non-Hamiltonian 3--regular, 3--edge-connected graph with girth $g$. However, this upper bound is expected to be a large overestimate. Indeed, as is demonstrated for small girth in the following section, the smallest non-Hamiltonian 3--regular graphs appear to be typically 3--edge-connected.

The concept of a $(k,g)$--cage can now be generalised to a $(k,g,e)$--prison, specifically, a non-Hamiltonian $k$--regular, $e$--edge-connected graph with girth $g$ that contains the minimum possible number of vertices. Using equivalent notation to that of cages, the number $\bar{n}(k,g,e)$ is defined to be the number of vertices in a $(k,g,e)$--prison. This manuscript has given constructive proofs of existence of $(k,g,e)$--prisons for $k = 3$, and furthermore if, for a given $k^*$, a constructive procedure of generating $k^*$--regular graphs of arbitrary girth $g$ is discovered for any given $k^*$, then this manuscript also gives a constructive proof of existence of $(k^*,g,1)$--prisons.

\section{Small prisons}\label{sec-prisons}

To conclude, the sizes of $(3,g,e)$--prisons are given for small $g$. Empty spaces correspond to as-of-yet unknown values. Values with a + sign indicate the prison is not unique.

\begin{table}[h!]
\begin{center}
\begin{tabular}{|l|c|c|c|c|c|}
  \hline
  \;\;\;\;\;$g$ & {\bf 3} & {\bf 4} & {\bf 5} & {\bf 6} & {\bf 7} \\
  $e$ & & & & & \\ \hline
  {\bf 1} & 10 & 14 & 22 & 30 & \\ \hline
  {\bf 2} & 14+ & 16 & 20 & 42+ & \\ \hline
  {\bf 3} & 12 & 14 & 10 & 28 & 28 \\ \hline
\end{tabular}
\caption{Values of $\bar{n}(3,g,e)$, the number of vertices in a $(3,g,e)$--prison.} \label{tab-prisons}
\end{center}
\end{table}

The values in Table \ref{tab-prisons} give rise to the following conjecture.

\vspace*{0.5cm}\begin{conjecture}For all $g \geq 4$, the value $n(3,g,3) \leq n(3,g,e)$ for $e = \{1, 2\}$.\end{conjecture}

Several of the prisons with small girth are known in literature. The $(3,3,3)$--prison is Flower snark ${\bf J}_3$ \cite{flower} (also known as Tietze's graph), the $(3,5,3)$--prison is the Petersen graph \cite{petersen}, the $(3,6,3)$--prison is the Flower snark ${\bf J}_7$ \cite{flower}, and the $(3,7,3)$--prison is the Coxeter graph \cite{coxeter}. The $(3,6,1)$--prison arises from two copies of the Heawood graph \cite{harary}, while at least one $(3,6,2)$--prison arises from the Heawood graph and the Flower snark ${\bf J}_7$. All other prisons displayed in Figure \ref{figur} arise from combinations of $K_4$, $K_{3,3}$ and the Petersen graph.

There are exactly two $(3,3,2)$--prisons. It is not currently known how many $(3,6,2)$--prisons exist, though some may be obtained by breaking an edge in both the Heawood graph and the Flower snark ${\bf J}_7$ and joining the two sets of broken edges together to form a 2-bond. All 3--regular graphs with girth 8 containing up to and including 44 vertices were checked, and all are Hamiltonian, so a $(3,8,e)$--prison must contain somewhere between 46 and 62 vertices (the latter bound is obtained by using the bridge construction in Section \ref{sec-bridge} on two copies of the Tutte-Coxeter graph \cite{tutte} on 30 vertices.)

Future work on this topic might seek to address not only edge-connectivity, but cyclic edge-connectivity, that is, the size of the minimal edge-cutset that separates a graph into multiple components each containing a cycle. It is worth noting that the each of the $(3,5,3)$, $(3,6,3)$ and $(3,7,3)$--prisons have cyclic connectivity equal to their girth, leading to the second conjecture of this manuscript.

\vspace*{0.5cm}\begin{conjecture}The cyclic edge-connectivity of a $(3,g,3)$--prison is equal to $g$, for all $g \geq 5$.\end{conjecture}

It is easy to see that the non-Hamiltonian 3-regular, 3-connected graphs produced by the construction in Section \ref{sec-non-bridge} of this manuscript always have cyclic connectivity of 3. Whether or not non-Hamiltonian 3-regular graphs exist with arbitrarily large cyclic connectivity exist at all is still an open problem.

\begin{figure}[htp!]
  \centering
      \label{figur}\caption{Examples of prisons for small girths.}

  \subfigure[The unique $(3,3,1)$--prison]{\label{figur:11}\includegraphics[scale=0.45]{prison331.eps}}
  \subfigure[The first $(3,3,2)$--prison]{\label{figur:12}\includegraphics[scale=0.45]{prison332_1.eps}}
  \\
  \subfigure[The second $(3,3,2)$--prison]{\label{figur:13}\includegraphics[scale=0.45]{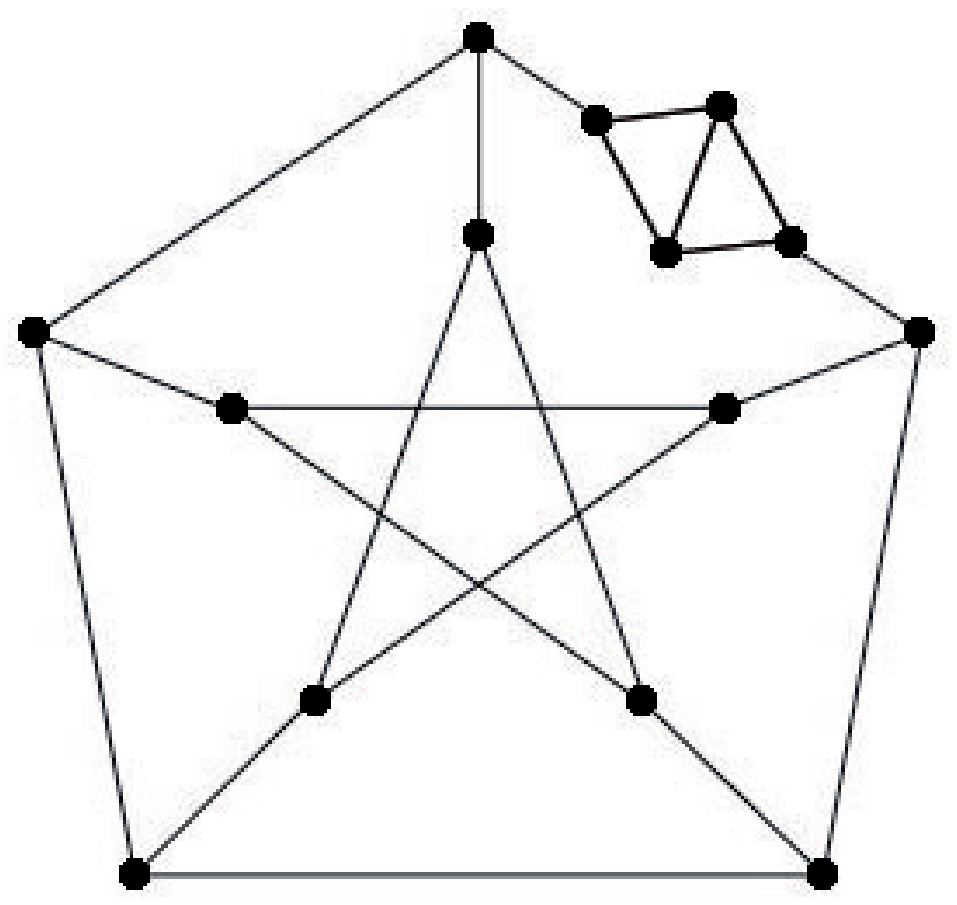}}
  \subfigure[The unique $(3,3,3)$--prison]{\label{figur:21}\includegraphics[scale=0.45]{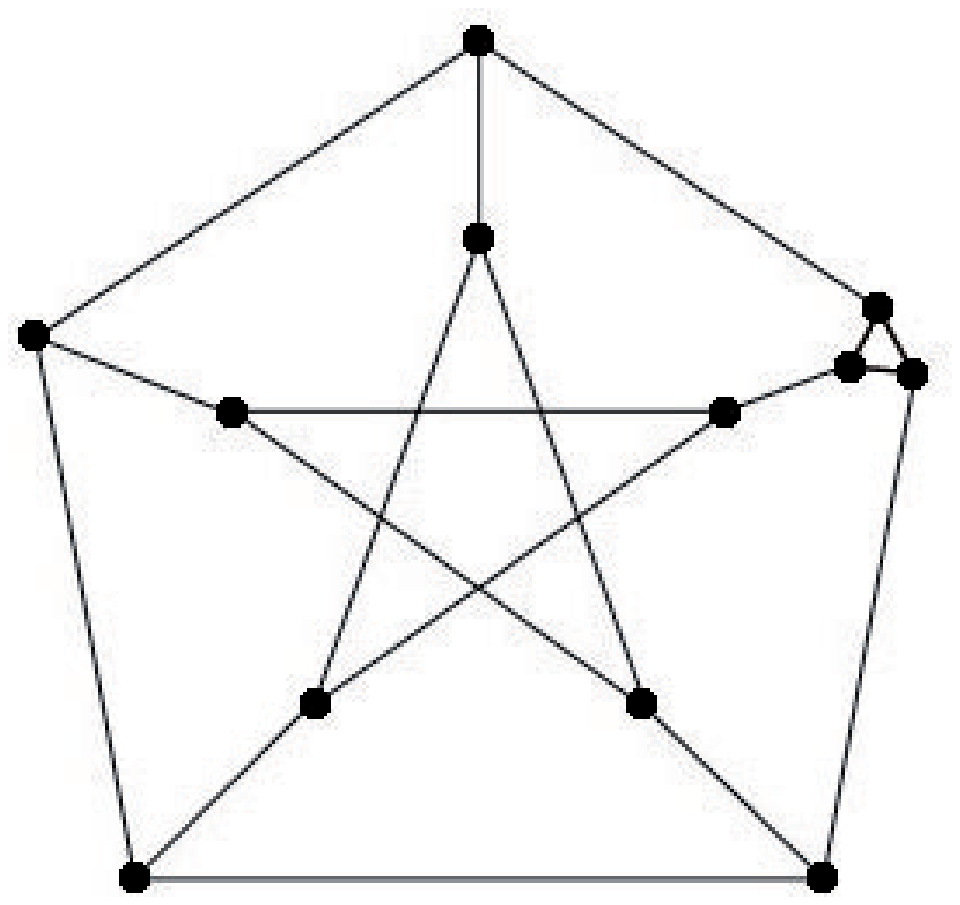}}
\\
  \subfigure[The unique $(3,4,1)$--prison]{\label{figur:22}\includegraphics[scale=0.45]{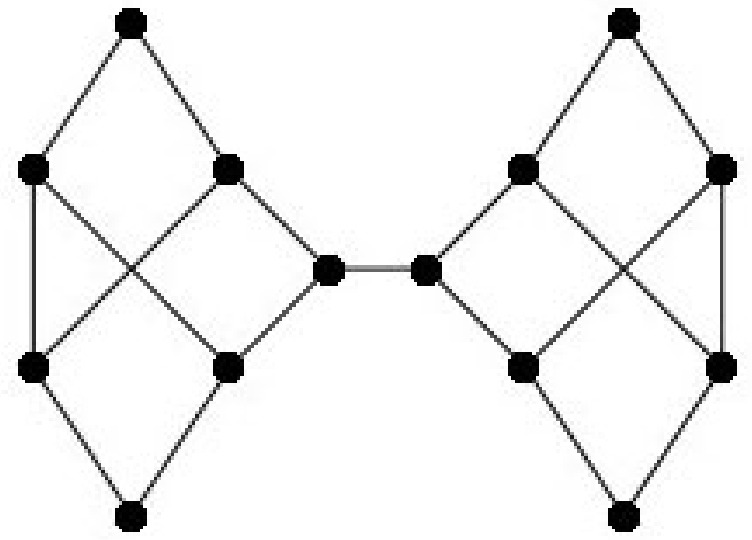}}
  \subfigure[The unique $(3,4,2)$--prison]{\label{figur:23}\includegraphics[scale=0.45]{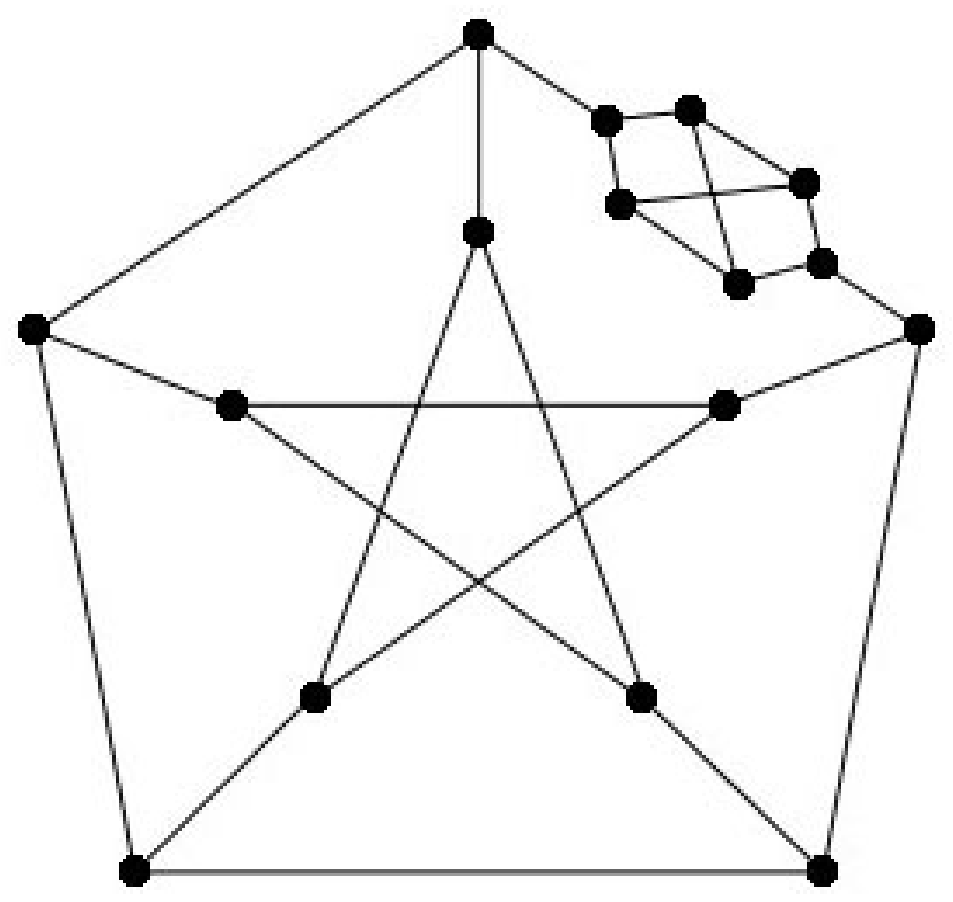}}
\\
    \subfigure[The unique $(3,4,3)$--prison]{\label{figur:31}\includegraphics[scale=0.45]{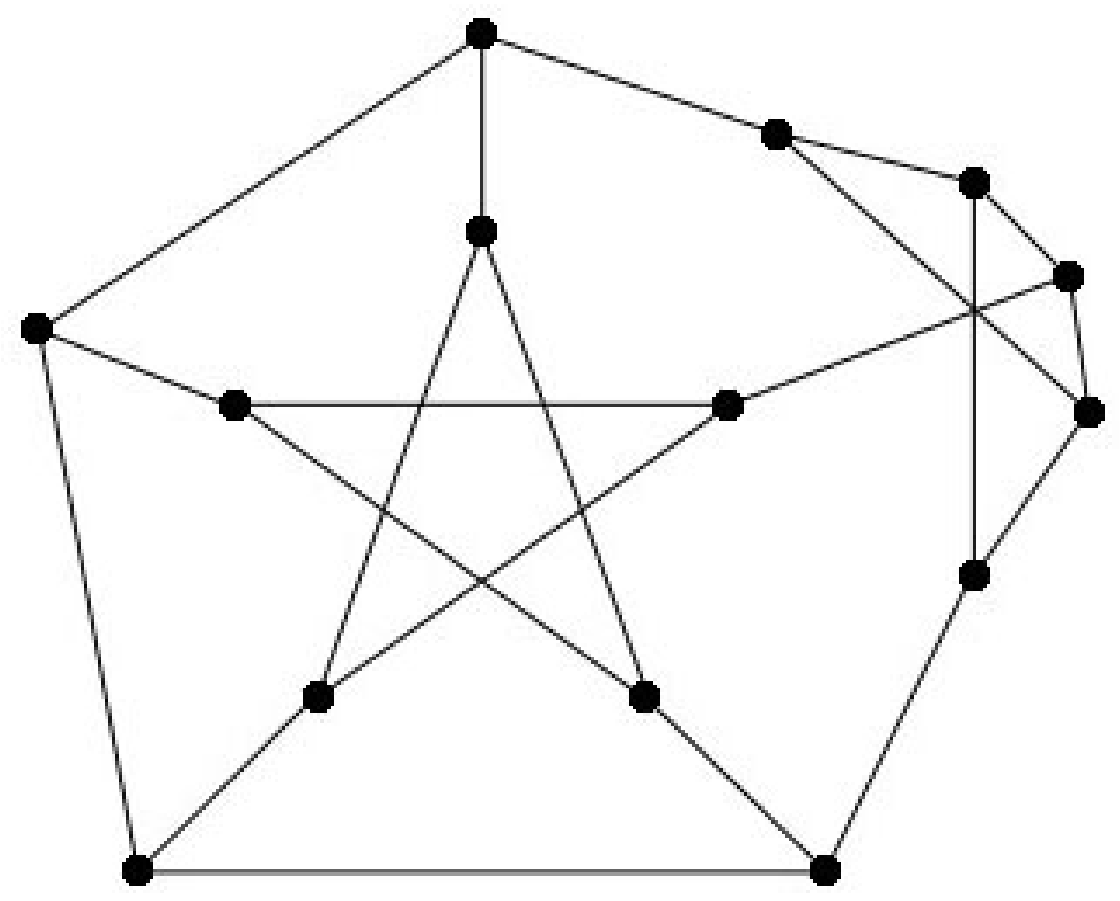}}
  \subfigure[The unique $(3,5,1)$--prison]{\label{figur:32}\includegraphics[scale=0.45]{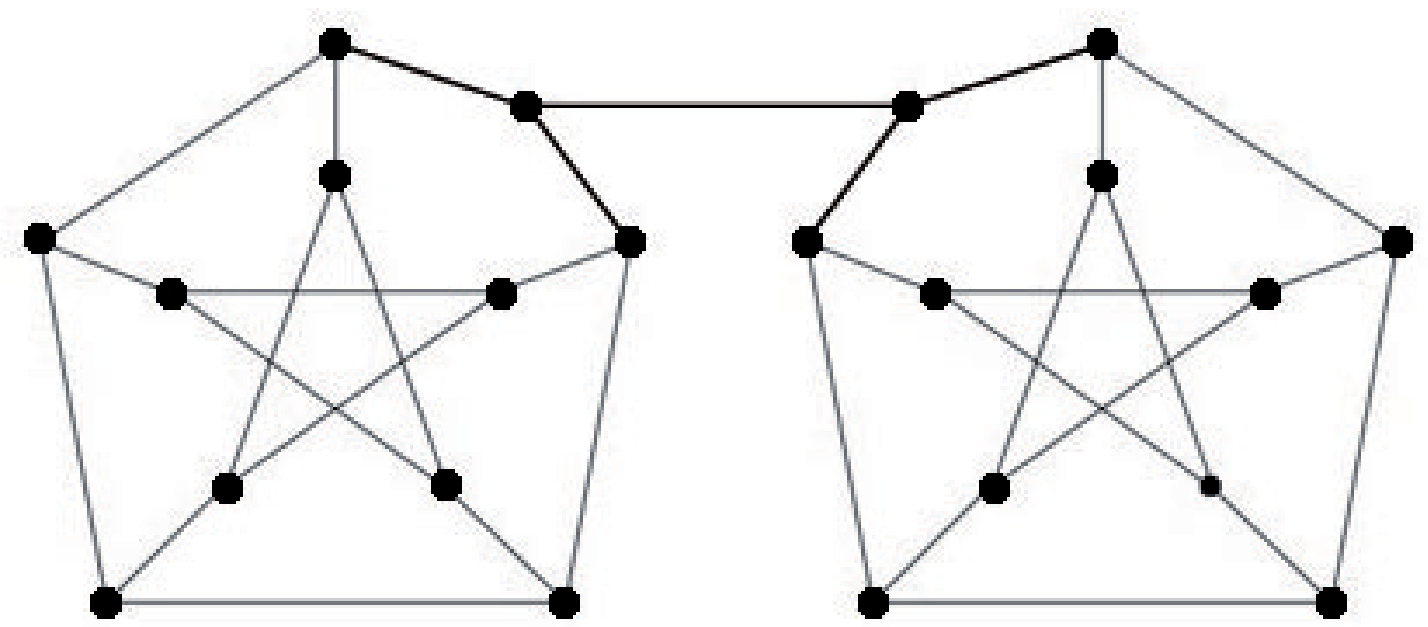}}
\end{figure}

\begin{figure}[htp!]
       \centering
        \label{figur}\caption{Examples of prisons for small girths.}
  \subfigure[The unique $(3,6,1)$--prison]{\label{figur:42}\includegraphics[scale=0.4]{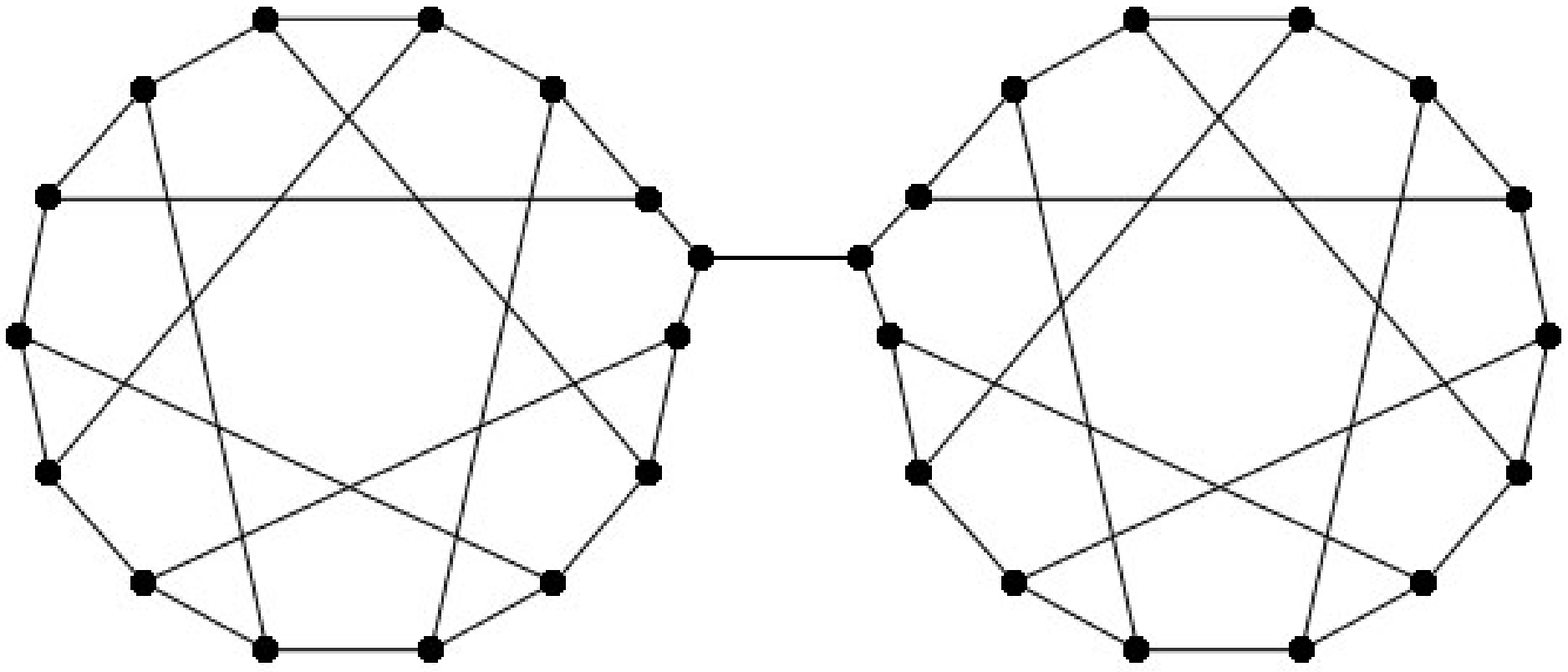}}
  \hspace*{0.45cm}\subfigure[The unique $(3,5,2)$--prison]{\label{figur:33}\includegraphics[scale=0.45]{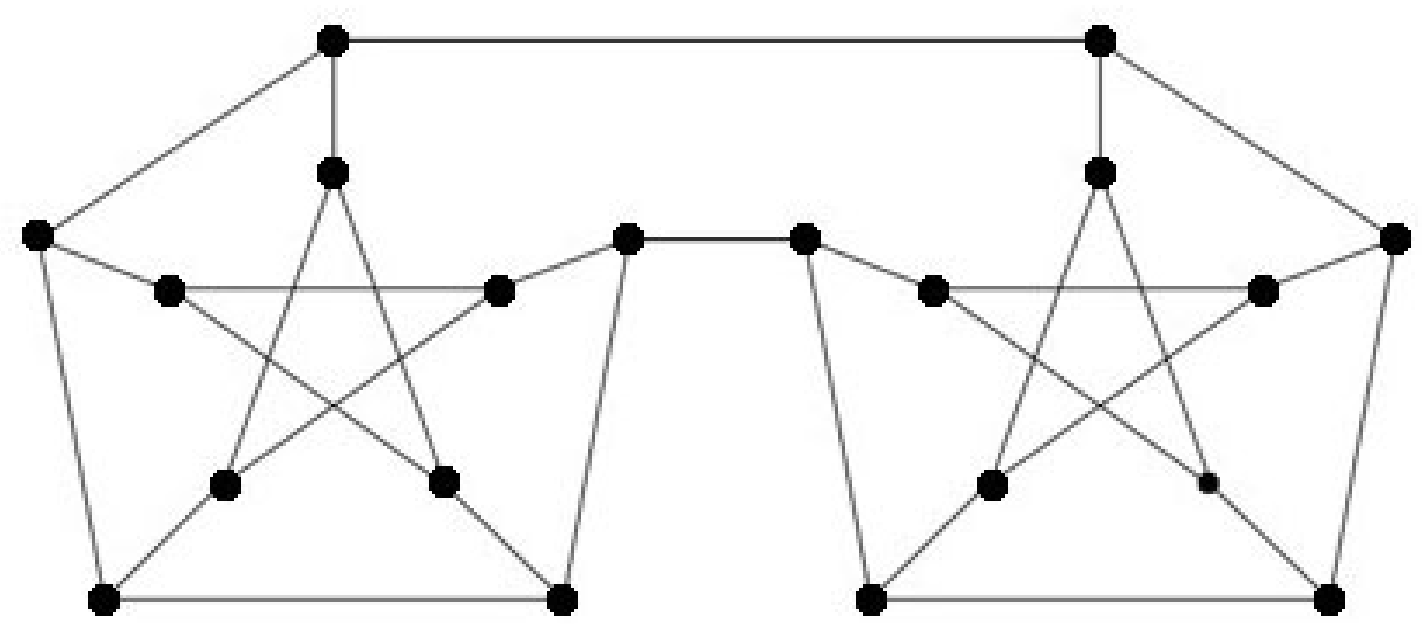}}
  \subfigure[The unique $(3,5,3)$--prison]{\label{figur:41}\includegraphics[scale=0.45]{prison353.eps}}
  \\
% \end{figure}
% \newpage
%  \begin{figure}[htp!]
%  \centering
%   \label{figur}\caption{Examples of prisons for small girths.}
  \subfigure[One example of a $(3,6,2)$--prison]{\label{figur:42}\includegraphics[scale=0.45]{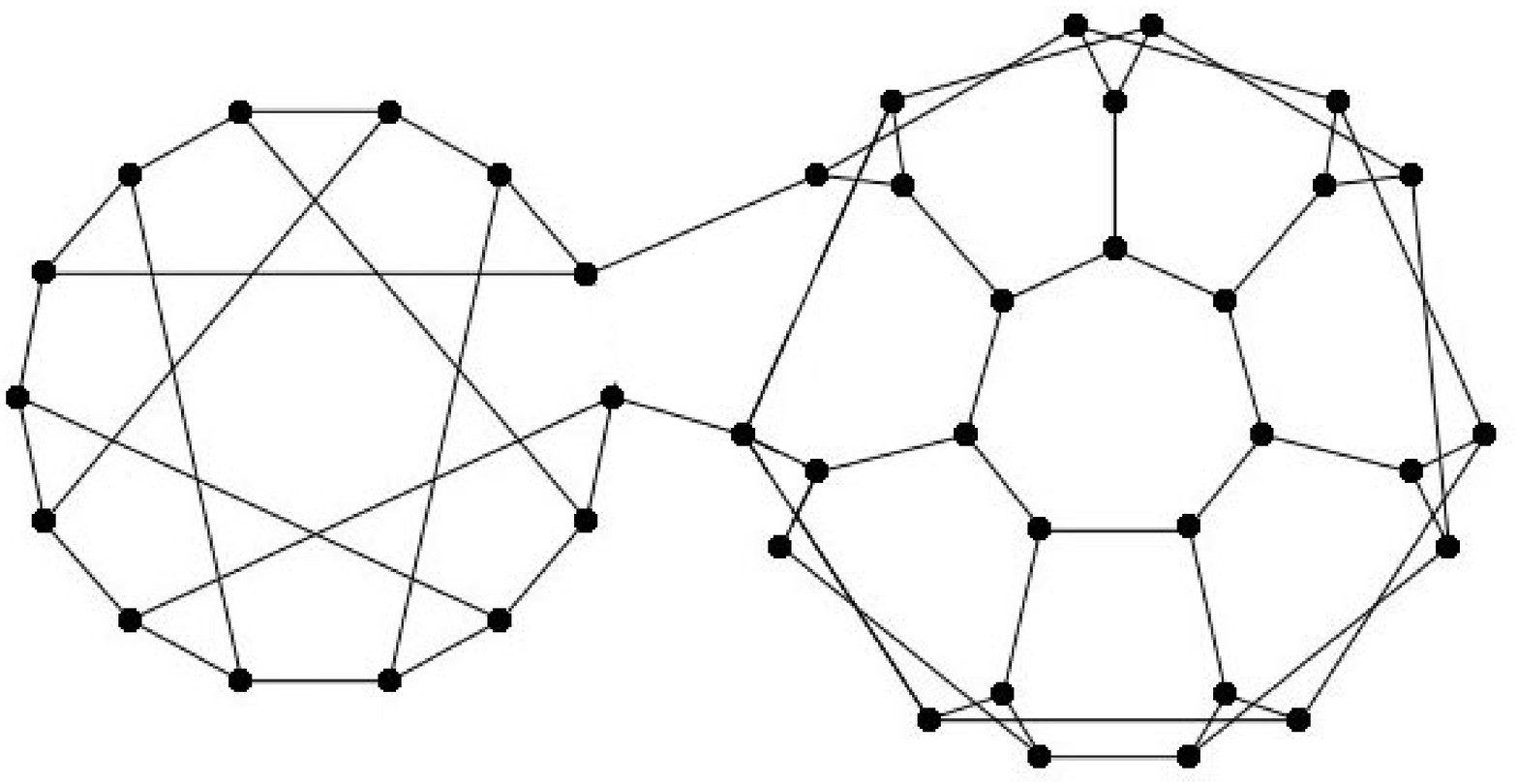}}\\
      \subfigure[The unique $(3,6,3)$--prison]{\label{figur:43}\includegraphics[scale=0.45]{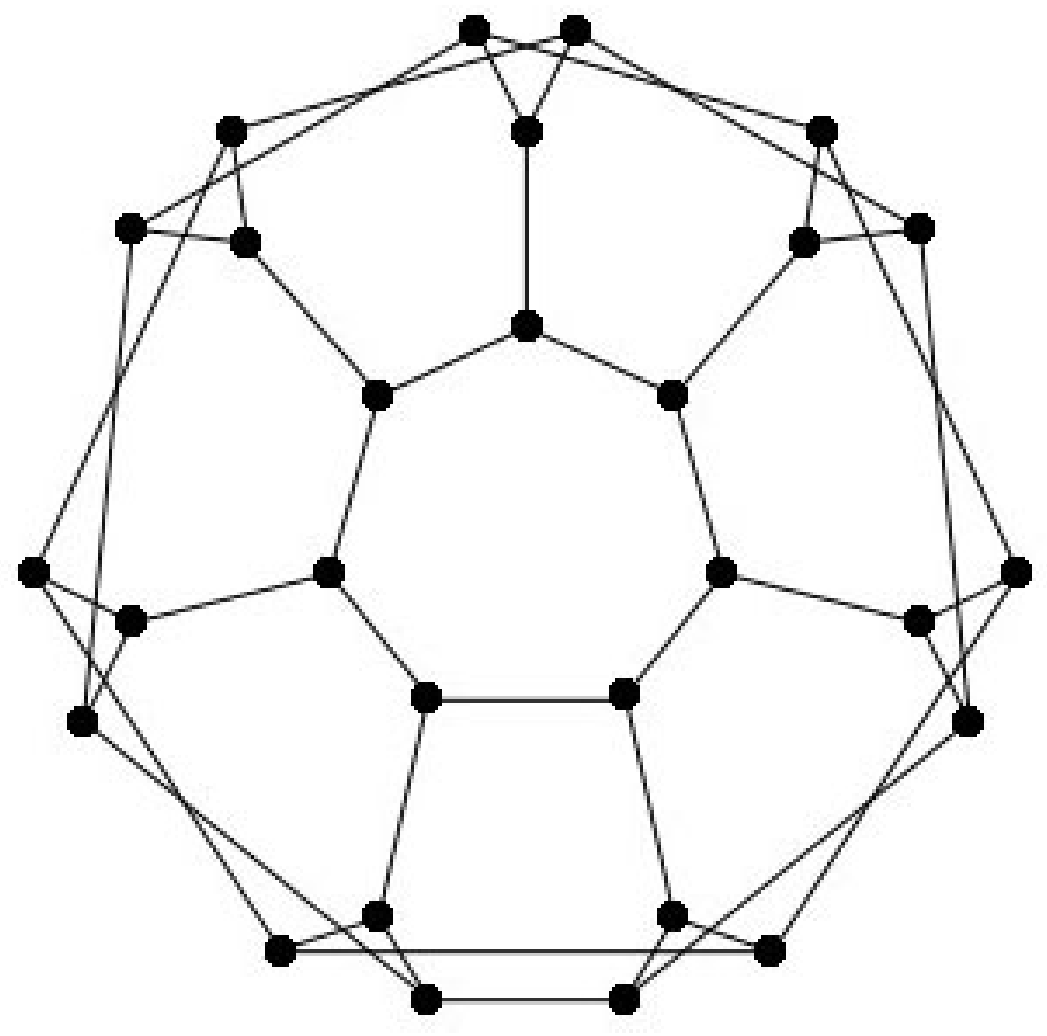}}
  \subfigure[The unique $(3,7,3)$--prison]{\label{figur:51}\includegraphics[scale=0.45]{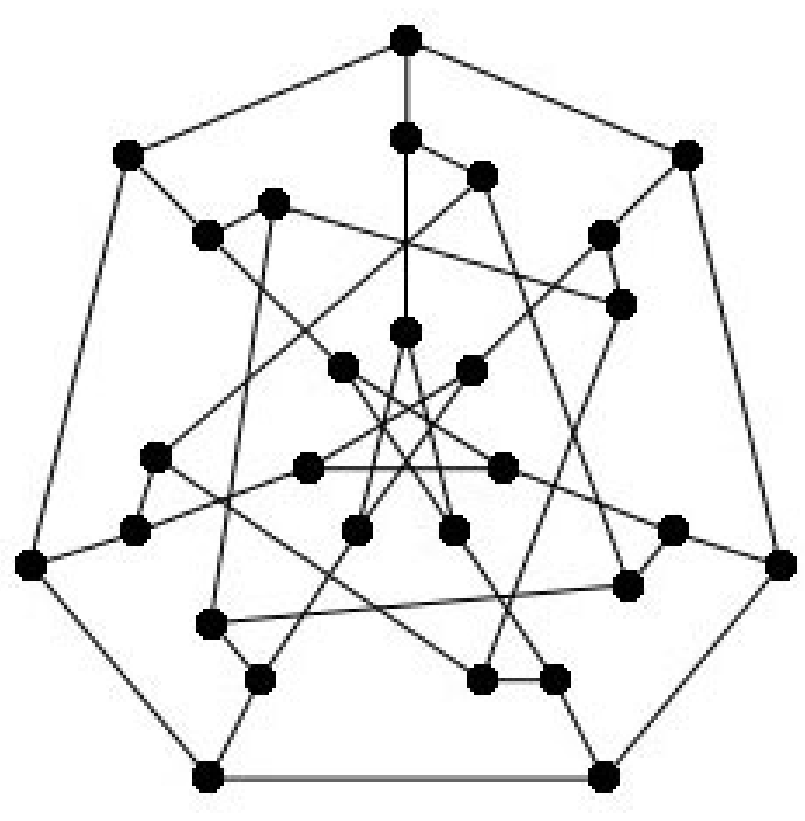}}\\
\end{figure}

\end{document}